\documentclass[reqno,11pt]{amsart}
\usepackage{amsmath,amsfonts,mathrsfs,amssymb}

\newcommand{\Z}{{\mathbb Z}}
\newcommand{\R}{{\mathbb R}}

\newtheorem{thm}{Theorem}[section]
\newtheorem{cor}[thm]{Corollary}
\newtheorem{prop}[thm]{Proposition}
\newtheorem{rem}[thm]{Remark}
\newtheorem{lem}[thm]{Lemma}

\begin{document}
\title{Iteration of  polynomial pair under Thue-Morse dynamic }

\author{Qinghui LIU}
\address[Q.H. LIU]{
Dept. Computer Sci.,
Beijing Institute of Technology,
Beijing 100081, PR China.}
\email{qhliu@bit.edu.cn}
\author{Yanhui QU}
\address[Y.H. QU]{Dept. Math., Tsinghua University, Beijing 100084, PR China.}
\email{yhqu@math.tsinghua.edu.cn}

\begin{abstract}
We study the behavior of a polynomial sequence which is defined by iterating a polynomial pair under Thue-Morse dynamic. We show that in suitable sense, the sequence will behave like $\{2\cos 2^nx: n\ge 1\}$. Basing on  this property
we can show that  the Hausdorff dimension of the spectrum  of the Thue-Morse Hamiltonian  has a common positive lower bound for all coupling.
\end{abstract}

\maketitle

\section{Introduction}

The trace polynomials related to Thue-Morse sequence has been studied since 1980s.  See especially the early works \cite{AP1,AP,B}. Let us recall the definitions.
Consider the Thue-Morse substitution
$$
\begin{cases}
\sigma(a)=ab\\
\sigma(b)=ba.
\end{cases}
$$
We denote the free group generated by $a,b$ as ${\rm FG}(a,b).$ Given $\lambda, x\in \R,$ we can define a homomorphism $\tau:{\rm
FG}(a,b)\to {\rm SL}(2,\R)$ as
$$
\tau(a)= \left[
\begin{array}{cc}
x-\lambda&-1\\
1&0
\end{array}
\right]\ \ \ \text{ and }\ \ \ \tau(b)= \left[
\begin{array}{cc}
x+\lambda&-1\\
1&0
\end{array}
\right]
$$
and $\tau(a_1\cdots a_n)=\tau(a_n)\cdots\tau(a_1).$   Define
$
h_n(x):={\rm tr}(\tau(\sigma^n(a)))
$
(where ${\rm tr}(A)$ denotes the trace of the matrix $A$),
by a direct computation we have
\begin{eqnarray}\label{recurrence}
\nonumber h_1(x)&=&x^2-\lambda^2-2;\ \  h_2(x)=(x^2-\lambda^2)^2-4x^2+2;\\
h_{n+1}(x)&=&h_{n-1}^2(x)(h_n(x)-2)+2 \ \  (n\ge2).
\end{eqnarray}
$\{h_n: n\ge 1\}$ is called the sequence of {\it trace polynomials} related to Thue-Morse sequence.

\eqref{recurrence} motivates the following definition.  Define $\Phi: \R^2\to\R^2$ as
$$
\Phi(x,y)=\left(y^2(x-2)+2, x\right).
$$
Then the recurrence relation \eqref{recurrence} is equivalent to $(h_{n+1},h_n)=\Phi(h_n,h_{n-1})$. We thus call  $\Phi$  the {\it Thue-Morse dynamic}.

In general, starting from a polynomial pair $(P_{-1},P_0)$, we can define
$$
(P_{n},P_{n-1}):=\Phi^n(P_0,P_{-1}).
$$
The main goal of this paper is to understand the behavior of the polynomial
sequence $\{P_n:n\ge -1\}$. Let us start with  one simple situation.

\subsection{A special sequence}\

 We take $\lambda=0$ and consider the sequence $\{h_n: n\ge1\}$.
   In this case $h_1(x)=x^2-2$ and
 $$
 h_2(x)=x^4-4x^2+2=h_1^2(x)-2=h_1\circ h_1(x).
 $$
By induction it is easy to show that
$$
h_n(x)=\underbrace{h_1\circ\cdots\circ h_1}_n(x).
$$
Thus in this special case, the iterations of $(h_1,h_2)$ are quite clear.

Observe  that if $\{P_n: n\ge-1\}$ is defined by the Thue-Morse dynamic, then for any change of variable $x=\varphi(y)$, the sequence $\{Q_n=P_n\circ \varphi:n\ge -1\}$ still satisfies the Thue-Morse dynamic.
 Of course, now $Q_n$ need not to be a polynomial.
 If we do the change of variable $x=2\cos y$,  then
by a simple computation we get
$$
g_n(y):=h_n(2\cos y)=2\cos (2^n y).
$$
Thus we conclude that $\{2\cos 2^nx: n\ge 1\}$ satisfies the Thue-Morse dynamic.

\subsection{General pictures}\label{gene-pic}\

A natural question is that what does $\{h_n: n\ge1\}$ look like when $\lambda\ne 0?$ More generally, starting from any polynomial pair $(P_{-1},P_0)$, and defining  $P_n$ according to Thue-Morse dynamic, what is the  behavior of $P_n?$

In this paper we will answer this question partly. Roughly speaking, let  $Z:=\{x: P_k(x)=0 \text{ for some } k\}$, then we will show that under some mild condition, around each $x\in Z,$ after suitable renormalization,  the sequence $\{P_n: n\ge-1\}$ will behave like $\{2\cos 2^nx: n\ge 1\}.$
We will make this precise in Section \ref{sec-2}.

\subsection{Application to  Thue-Morse Hamiltonian}\

This general result can be applied  to the spectral problem of Thue-Morse Hamiltonian and give a uniform lower bound for the Hausdorff dimension of the spectrum.

Let us recall the definition of   discrete Schr\"odinger operator. Given a bounded real sequence $v=\{v(n)\}_{n\in\Z}$ and $\lambda\in\R$, we can  define an operator $H_{\lambda,v}$  act on $l^2(\mathbb{Z})$
as
$$ 
(H_{\lambda,v}\psi)(n)=\psi(n+1)+\psi(n-1)+\lambda v(n)\psi(n),\ \forall n\in\mathbb{Z}.
$$ 
$H_{\lambda,v}$ is called an {\it discrete Schr\"odinger operator with potential $\lambda v$}; $\lambda$ is called the {\it coupling} constant. We denote the spectrum of $H_{\lambda,v}$ by $\sigma(H_{\lambda,v})$.

We define the two-sided  Thue-Morse sequence   $v$ as follows:  Let $\sigma$ be the Thue-Morse sbustitution, let $u=u_1u_2\cdots:=\sigma^\infty(a).$ For $n\ge 1$, let  $v(n)=1$ if $u_n=a$; let $v(n)=-1$
 if $u_n=b;$ let $v(1-n)=v(n)$ for $n\ge 1.$ The operator $H_{\lambda,v}$
 with   Thue-Morse sequence $v$ is called {\it Thue-Morse Hamiltonian}.
   We will prove the following theorem.

\begin{thm}\label{main-bd}
There exists an absolute constant $C>0$  such that for  Thue-Morse sequence $v$  and any $\lambda\in\R$,
 $$
 \dim_H\sigma(H_{\lambda,v}) \ge C.
 $$
 \end{thm}

 \begin{rem}
 {\rm
 Axel and Peyri\`ere(\cite{AP1,AP})
study the spectrum, and prove numerically that its Lesbesgue measure is
zero, and the Box dimension is strictly less than $1$.
Then it is rigorously  proven that the spectrum is a Cantor set of Lebesgue measure zero
(see for example \cite{B,BG,Lenz,LTWW}). By our best knowledge, no rigorous results about the Hausdorff dimension of the spectrum has been proven before.
}
\end{rem}

\begin{rem}
{\rm
By our result, the dimension property of Thue-Morse Hamiltonian   is quite
different from another heavily studied model -- the Fibonacci Hamiltonian.
Recall that the Fibonacci sequence $w$ is defined by
$$
w(n)=\chi_{[1-\alpha,1)}(n\alpha \mod 1),\quad \forall n\in\mathbb{Z}
$$
with $\alpha=(\sqrt{5}+1)/2.
$
The Fibonacci Hamiltonian $H_{\lambda,w}$ is a  central model in discrete Schr\"odinger  operator.  Its dimensional properties has been extensively studied, see for example \cite{R,JL,LW,DEGT,C,DG,DG2}. In purticular  the following property is shown in \cite{DEGT}:
$$
\lim_{|\lambda|\to\infty} \dim_H \sigma(H_{\lambda,w})\ln  |\lambda|=\ln (1+\sqrt{2}).
$$
This implies that $\dim_H \sigma(H_{\lambda,w})\to 0$ with the speed $1/\ln |\lambda|$ when $|\lambda|\to\infty.$ However by our result, the dimension 
of spectrum for Thue-Morse potential (i.e., $\sigma(H_{\lambda,v})$) 
has a uniform positive lower bound.
}
\end{rem}

\begin{rem} \label{rem-zero}
{\rm
The rough idea of proving Theorem \ref{main-bd} is the following. Recall that $\{h_n:n\ge 1\}$ is the trace polynomial sequence  related to the Thue-Morse sequence. Define
$$
\Sigma=\{x\in\mathbb{R}: \exists n>0, h_n(x)=0\}.
$$
It is shown in \cite{AP,B} that $\Sigma\subset \sigma(H_{\lambda,v})$. Basing  on the general picture described in Section \ref{gene-pic}, it is possible to construct a Cantor subset $\mathcal C$ of $\sigma(H_{\lambda,v})$ in a controllable fashion, then we can estimate the Hausdorff dimension of $\mathcal C$, which in turn offer a lower bound for the dimension of the spectrum.
}
\end{rem}

The rest of the paper is organized as follows.
In Section \ref{sec-2},  we show that the polynomial sequence will behave like $\{2\cos 2^nx:n\ge 1\}$ near a base point of a germ.
In Section \ref{sec-3}, we prepare the proof of the lower bound of the spectrum.  In Section \ref{sec-4}, we prove Theorem \ref{main-bd}.

\section{Convergence towards $\{2\cos 2^nx: n\ge 1\}$}\label{sec-2}

Given polynomial pair $(f_{-1},f_0)$. We recall that defining  $\{f_n:n\ge 1\}$ according to $(f_n, f_{n-1}):=\Phi^n(f_0,f_{-1})$ for $n\ge 0$ is equivalent to define it according to the recurrence relation
$f_{n+1}=f_{n-1}^2(f_n-2)+2$.

\subsection{Germ of a polynomial pair }\

Given polynomial pair $(f_{-1},f_0)$. Define $f_{n+1}=f_{n-1}^2(f_n-2)+2$ for $n\ge 0.$ Assume $f_0(x_0)=0$, at first we study the local behavior of $f_n$ at $x_0.$
Write
 $$
 f_0(x)=f^\prime(x_0)(x-x_0)+O((x-x_0)^2)\ \text{ and }\ f_1(x)=f_1(x_0)+O((x-x_0)).
 $$
 By the recurrence relation we have
 \begin{eqnarray*}
 f_2(x)&=&2-(2-f_1(x_0))f_0^{\prime2}(x_0)(x-x_0)^2 +O((x-x_0)^3)\\
  f_k(x)&=&2-4^{k-3}(2-f_1(x_0))\left(f_0^{\prime}(x_0)f_1(x_0)\right)^2(x-x_0)^2 \\
  && +O((x-x_0)^3)\ \ \ \ (k\ge 3).
 \end{eqnarray*}
 If moreover $f_1(x_0)< 2$ and $f_0^\prime(x_0), f_1(x_0)\ne 0$, then
 $$
 \rho:= \sqrt{2-f_1(x_0)}|f_0^{\prime}(x_0)f_1(x_0)|>0
 $$ and   for $k\ge 3$
 \begin{equation}\label{fkx}
f_k(x)=2-(2^{k-3}\rho)^2(x-x_0)^2+O((x-x_0)^3).
\end{equation}
If we define $\tilde f_k(x)=f_{k+3}(x/\rho+x_0)$, then for $k\ge 0$
$$
\tilde f_k(x)= 2-(2^{k}x)^2+O(x^3).
$$
Notice that we also have
$$
2\cos 2^kx =2-(2^kx)^2+O(x^3).
$$
Thus $\{\tilde f_k (x):k\ge 1\}$ is a good candidate of polynomial sequence which converge to $\{2\cos 2^kx: k\ge 1\}$. This also motivates the following definition.

Given a polynomial pair  $(P_{-1}, P_0)$. Assume at  $x_0\in\mathbb{R}$,
 there exists $\rho>0$ such that
$$
 \begin{cases}
 P_{-1}(x)&=2-\frac{\rho^2}{4}(x-x_0)^2+O((x-x_0)^3);\\
  P_{0}(x)&=2-\rho^2(x-x_0)^2+O((x-x_0)^3)
 \end{cases}
$$
Then we say $(P_{-1},P_0)$ has a {\it  $\rho$-germ } at $x_0$. $x_0$ is called the {\it base point} of the germ.

Assume $(P_{-1}, P_0)$ has a $\rho$-germ at $x_0$.
For $k\ge1$, define
\begin{equation}\label{iter}
P_k=P_{k-2}^2(P_{k-1}-2)+2.
\end{equation}
For $k\ge -1$ define
\begin{equation}\label{iterscale}
Q_k(x)=P_k(\frac{x}{2^k\rho}+x_0).
\end{equation}
It is ready to show that
 $
 Q_{k}(x)=2-x^2+O(x^3).
 $
Since $2\cos x=2-x^2+O(x^3)$,  we conclude taht   $Q_k(x)=2\cos x +O(x^3).$
Write $\Delta_k(x)= Q_k(x)-2\cos x,$ then
\begin{equation}\label{delta-k}
\Delta_k(x)=Q_k(x)-2\cos x=\sum_{k\ge 3} \Delta_{k,n}x^n.
\end{equation}

\subsection{Regular germ of polynomial pair}\

Our goal is to show that $\Delta_k(x)\to 0$ for $x$ in any bounded interval,  to achieve this we need to impose some condition on the initial pair $(P_{-1},P_0)$, or equivalently on $(Q_{-1},Q_0)$. Our condition is about the coefficients of $\Delta_{-1}$ and $\Delta_0.$ Let us do some preparation.

Given two formal series with real coefficients
$$
f(x)=\sum_{n=0}^\infty a_nx^n \ \ \ \ \text{ and }\ \ \ \ g(x)=\sum_{n=0}^\infty b_nx^n.
$$
If we only concern about their coefficients, we can define the following partial order:
$$
f\preceq g  \Leftrightarrow a_n\le b_n  \ \ (\forall n\ge 0).
$$
We further  define $|f(x)|^\ast:=\sum_{n=0}^\infty |a_n|x^n.$ Then it is easy to check that
$$
|fg|^\ast\preceq |f|^\ast|g|^\ast\ \ \ \text{ and }\ \ |f+g|^\ast\preceq |f|^\ast+|g|^\ast.
$$
Moreover if $|f|^\ast\preceq \tilde f$ and $|g|^\ast\preceq \tilde g$, then it is seen that $|fg|^\ast\preceq \tilde f\tilde g.$
Later we will use these properties repeatedly to estimate the coefficients of certain series.

Let us go back to $\{P_k: k\ge -1\}$ discussed above.
 If moreover there exist $\delta>0$ and $\beta\ge1$ such that
$$
|\Delta_{-1}|^\ast, |\Delta_0|^\ast\preceq \delta\sum_{n=3}^\infty \frac{x^n}{\beta^n}
$$
Then we say that $(P_{-1},P_0)$  has a  {\it  $(\delta,\beta)$-regular $\rho$-germ at $x_0$}.  We also say that $(P_{-1},P_0)$  is  {\it  $(\delta,\beta)$-regular at $x_0$
with renormalization factor $\rho$}, or simply  as {\it  $(\delta,\beta)$-regular at $x_0$}.

With this definition, now we can state our main convergence theorem.

\begin{thm}\label{exp-cvge}
Assume  $(P_{-1},P_0)$ has a $(1,1)$-regular $\rho$-germ at $x_0$. Define
 $(P_k)_{k\ge-1}$
and  $\{Q_k\}_{k\ge-1}$ according to \eqref{iter} and \eqref{iterscale} respectively. Let $\Delta_k(x)=Q_k(x)-2\cos x$.  Then for any $m\ge0$, there exists an absolute constant $C_m>0$ such that for any  $k\ge 2m+1$ and any $x\in[-2^{m-1}\pi,2^{m-1}\pi]$,
$$
|\Delta_k(x)|
\le \tilde C_m \alpha^k|x|^3\le C_m\alpha^k.
$$
\end{thm}

\begin{rem}
{\rm
This theorem says that in any bounded interval, the polynomial sequence $Q_k(x)$ 
will converge to $2\cos x$ uniformly with exponential speed.
If we define $\tilde P_k(x):=Q_k(2^kx)$, then  the sequence $\{\tilde P_k(x):k\ge 1\}$ will behave like $\{2\cos 2^kx: k\ge 1\}$ locally. On the other hand, notice that by \eqref{iterscale}, we have 
$$
\tilde P_k(x)=Q_k(2^kx)=P_k(\frac{x}{\rho}+x_0).
$$
Thus $\tilde P_k$ is a renormalization of $P_k$. As a conclusion, the renormalization of the sequence $\{P_k:k\ge 1\}$ behave like $\{2\cos 2^kx: k\ge 1\}$ locally, which gives a precise version of Section \ref{gene-pic}.

}
\end{rem}

\subsection{Convergence properties}\

We prove Theorem \ref{exp-cvge} at the end of this subsection.
By the recurrence relation \eqref{iter}, we have for $k\ge 1$
\begin{eqnarray*}
Q_k(x)
&=&Q_{k-2}^2(x/4)(Q_{k-1}(x/2)-2)+2\\
&=&\Big(2\cos x/4+\Delta_{k-2}(x/4)\Big)^2\Big(2\cos x/2-2+\Delta_{k-1}(x/2)\Big)+2\\
&=&2\cos x+(2+2\cos \frac{x}{2})\cdot\Delta_{k-1}(\frac{x}{2})+
\\&&\Delta_{k-2}(\frac{x}{4})\Big(4\cos \frac{x}{4}+\Delta_{k-2}
(\frac{x}{4})\Big)\Big(2\cos \frac{x}{2}-2+\Delta_{k-1}(\frac{x}{2})\Big).
\end{eqnarray*}
Thus we conclude that for $k\ge 1$
\begin{eqnarray}\label{delta-k-x}
\Delta_k(x)&=&(2+2\cos \frac{x}{2})\cdot\Delta_{k-1}(\frac{x}{2})+\\
\nonumber&&\Delta_{k-2}(\frac{x}{4})\Big(4\cos \frac{x}{4}+\Delta_{k-2}
(\frac{x}{4})\Big)\Big(2\cos \frac{x}{2}-2+\Delta_{k-1}(\frac{x}{2})\Big).
\end{eqnarray}

\begin{prop}\label{model}
Assume $
\Delta_k(x)=\sum_{k\ge 3} \Delta_{k,n}x^n
$ and  $(\Delta_k(x))_{k\ge-1}$ satisfy \eqref{delta-k-x}.
Fix $k\ge0$. If there exists $0<\delta\le1$ such that $|\Delta_{k-1,n}|,  |\Delta_{k,n}|\le \delta$ for any $n\ge3$,
then
$$
|\Delta_{k+1}(x)|^*\preceq
\delta\left(4\frac{x^3}{2^3}+4\frac{x^4}{2^4}+
9\sum\limits_{n\ge5}\frac{x^n}{2^n}\right).
$$
\end{prop}
\begin{proof}
We only need to prove the case $k=1$.
By \eqref{delta-k-x} we write
\begin{eqnarray*}
\Delta_1(x)&=&(2+2\cos \frac{x}{2})\cdot\Delta_0(\frac{x}{2})+\\
&&\Delta_{-1}(\frac{x}{4})\Big(4\cos \frac{x}{4}+\Delta_{-1}(\frac{x}{4})\Big)
\Big(2\cos \frac{x}{2}-2+\Delta_0(\frac{x}{2})\Big)\\
&=:&(I)+\Delta_{-1}(\frac{x}{4})(II)(III),
\end{eqnarray*}
where
$$
\begin{cases}
(I)&=(2+2\cos \frac{x}{2})\cdot\Delta_0(\frac{x}{2}),\\
(II)&=4\cos \frac{x}{4}+\Delta_{-1}(\frac{x}{4}),\\
(III)&=2\cos \frac{x}{2}-2+\Delta_0(\frac{x}{2}).
\end{cases}
$$
Since $\delta\le1$, we have
$$
\begin{array}{rcl}
|(I)|^*&\preceq&\left(4+2\sum\limits_{n\ge2}\frac{x^n}{n!2^n}\right)
\sum\limits_{n\ge3}\delta\frac{x^n}{2^n}\\
&=&4\delta\sum\limits_{n\ge3}\frac{x^n}{2^n}+
2\delta\sum\limits_{n\ge5}\frac{x^n}{2^n}\sum_{k=2}^{n-3}\frac{1}{k!}\\
&\preceq&4\delta\sum\limits_{n\ge3}\frac{x^n}{2^n}+
2\delta\sum\limits_{n\ge5}\frac{x^n}{2^n}(e-2)\\
&\preceq& 4\delta\frac{x^3}{2^3}+4\delta\frac{x^4}{2^4}+
6\delta\sum\limits_{n\ge5}\frac{x^n}{2^n}\\
|(II)|^*&\preceq&4+4\sum\limits_{n\ge2}\frac{x^n}{n!4^n}+
\delta\sum\limits_{n\ge3}\frac{x^n}{4^n}\preceq
4+2\sum\limits_{n\ge2}\frac{x^n}{4^n}\\
|(III)|^*&\preceq&\frac{x^2}{2^2}\frac{2}{2!}+
\sum\limits_{n\ge3}\frac{x^n}{2^n}\left(\frac{2}{n!}+\delta\right)
\preceq \frac{x^2}{2^2}+2\sum\limits_{n\ge3}\frac{x^n}{2^n}.
\end{array}
$$
Thus
$$
\begin{array}{rcl}
|\Delta_{-1}(\frac{x}{4})\times (II)|^*
&\preceq&\delta\sum_{n\ge3}\frac{x^n}{4^n}\left(4+2\sum\limits_{n\ge2}\frac{x^n}{4^n}\right)\\
&=&4\delta\sum\limits_{n\ge3}\frac{x^n}{4^n}+2\delta\sum\limits_{n\ge5}\frac{x^n}{4^n}(n-4)\\
&=&4\delta\frac{x^3}{4^3}+2\delta\sum\limits_{n\ge4}\frac{x^n}{4^n}(n-2)\\
|\Delta_{-1}(\frac{x}{4})\times (II)\times (III)|^*&
\preceq &
\frac{x^5}{2^5}\frac{\delta}{2}+\sum\limits_{n\ge6}\frac{x^n}{2^n}\frac{(n-4)\delta}{2^{n-3}}
+\\
&&\delta\sum\limits_{n\ge6}\frac{x^n}{2^n}+
\delta\sum\limits_{n\ge7}\frac{x^n}{2^n}\sum\limits_{k=4}^{n-3}\frac{k-2}{2^{k-2}}\\
&\preceq&\frac{x^5}{2^5}\frac{\delta}{2}+\frac{x^6}{2^6}\frac{5}{4}\delta+
3\delta\sum\limits_{n\ge7}\frac{x^n}{2^n}.
\end{array}
$$
This prove the proposition.
\end{proof}

To prepare the proof for the dimension of the spectrum, we also need to study a variant  of \eqref{delta-k-x}, i.e., for  $\tilde \Delta_k(x)=\sum_{n\ge0}\tilde \Delta_{k,n}x^n, (k\ge -1)$
there exists a constant $t_0$ such that for any $k\ge1$,
\begin{eqnarray}\label{delta-k-x-1}
\tilde \Delta_k(x)&=&(2+2\cos \frac{x+t_0}{2})\cdot\tilde \Delta_{k-1}(\frac{x}{2})+\tilde \Delta_{k-2}(\frac{x}{4})\cdot\\
\nonumber&&\Big(4\cos \frac{x+t_0}{4}+\tilde \Delta_{k-2}
(\frac{x}{4})\Big)\Big(2\cos \frac{x+t_0}{2}-2+\tilde \Delta_{k-1}(\frac{x}{2})\Big).
\end{eqnarray}

\begin{prop}\label{model2}
Assume $(\tilde \Delta_k(x))_{k\ge-1}$ satisfy \eqref{delta-k-x-1}.
Fix $k\ge0$. If there exist $0<\delta,\beta\le1$ such that
$|\tilde \Delta_{k-1,n}|,
|\tilde \Delta_{k,n}|\le \delta \beta^{-n}$ for any $n\ge0$,
then
$$
|\tilde \Delta_{k+1}(x)|^*\preceq 152\delta\sum_{n\ge0}\frac{x^n}{(2\beta)^n}.
$$
\end{prop}
\begin{proof}
We only need to prove the case $k=1$.
By \eqref{delta-k-x-1}, we can write
\begin{eqnarray*}
\tilde \Delta_1(x)&=&(I)+\tilde \Delta_{-1}(\frac{x}{4}))(II)(III).
\end{eqnarray*}
where
$$\begin{array}{l}
(I)=\Big(2+2\cos (\frac{x}{2}+\frac{t_0}{2})\Big)\cdot\tilde \Delta_0(\frac{x}{2}),\\
(II)=4\cos (\frac{x}{4}+\frac{t_0}{4})+\tilde \Delta_{-1}(\frac{x}{4}),\\
(III)=2\cos (\frac{x}{2}+\frac{t_0}{2})-2+\tilde \Delta_0(\frac{x}{2}).
\end{array}
$$
Note that
\begin{equation}\label{cosine}
|\cos(x+x_0)|^*=|\sum_{n\ge0}\frac{\cos(x_0+n\pi/2)}{n!}x^n|^\ast\preceq \sum_{n\ge0}\frac{x^n}{n!}.
\end{equation}
Since $\delta,\beta\le1$, we have
$$
\begin{array}{rcl}
|(I)|^*&\preceq&\Big(\sum\limits_{n\ge0}4\frac{x^n}{n!2^n}\Big)\Big(
\sum\limits_{n\ge0}\delta\frac{x^n}{2^n\beta^n}\Big)
\preceq
4\delta\sum\limits_{n\ge0}\frac{x^n}{(2\beta)^n}\sum_{k=0}^{n}\frac{\beta^k}{k!}\\
&\preceq& 4\delta e^\beta\sum\limits_{n\ge0}\frac{x^n}{(2\beta)^n}
\preceq12\delta\sum\limits_{n\ge0}\frac{x^n}{(2\beta)^n}\\
|(II)|^*&\preceq&4\sum\limits_{n\ge0}\frac{x^n}{n!4^n}+
\delta\sum\limits_{n\ge0}\frac{x^n}{(4\beta)^n}
=\sum\limits_{n\ge0}\frac{x^n}{4^n}
\left(\frac{4}{n!}+\frac{\delta}{\beta^n}\right)\\
|(III)|^*&\preceq&
\sum\limits_{n\ge0}\frac{x^n}{2^n}\left(\frac{4}{n!}+\frac{\delta}{\beta^n}\right)
\preceq (4+\delta)\sum\limits_{n\ge0}\frac{x^n}{(2\beta)^n}.
\end{array}
$$
Consequently
$$
\begin{array}{rcl}
|\tilde \Delta_{-1}(\frac{x}{4})\times (II)|^*&\preceq&
\delta\sum_{n\ge 0}\frac{x^n}{4^n}\sum_{k=0}^n\left[\frac{4}{k!}
+\frac{\delta}{\beta^k}\right]\frac{1}{\beta^{n-k}}\\
&\preceq&\delta\sum_{n\ge 0}\frac{x^n}{(4\beta)^n}(4e^\beta+(n+1)\delta)
\end{array}
$$

$$\begin{array}{rcl}
|\tilde \Delta_{-1}(\frac{x}{4})\times (II)\times (III)|^*&
\preceq &
\delta(4+\delta)\sum_{n\ge0}\frac{x^n}{(2\beta)^n}
\sum_{k=0}^n \frac{4e^\beta+(k+1)\delta}{2^k}.\\
&\preceq&\delta(4+\delta)(8e^\beta+4\delta)\sum_{n\ge0}\frac{x^n}{(2\beta)^n}\\
&\preceq&140\delta\sum_{n\ge0}\frac{x^n}{(2\beta)^n}.
\end{array}
$$
This prove the proposition.
\end{proof}

\begin{prop}\label{control-coefficient}
Let  $(P_{-1},P_0)$ be $(1,1)$-regular at $x_0$ with renormalization factor $\rho$,
and $(P_k)_{k\ge-1}$ satisfy \eqref{iter}. Define $(Q_k)_{k\ge-1}$, $(\Delta_k)_{k\ge-1}$ and $(\Delta_{k,n})_{k\ge-1,n\ge3}$ as in \eqref{iterscale} and \eqref{delta-k}.
Let  $\alpha=2^{-1/2}$, then for any $k\ge1$,
\begin{equation}\label{pk}
|\Delta_k(x)|^*\preceq 9\alpha^{k-2}\sum_{n\ge3}\frac{x^n}{2^n}.
\end{equation}
Consequently $(P_{k-1},P_{k})$ is $(9\alpha^{k-3},2)$-regular at $x_0$ for any $k\ge2$.
\end{prop}

\begin{proof}
By the condition we have
$$|\Delta_{-1,n}|\le 1,\quad |\Delta_{0,n}|\le 1,\quad \forall n\ge3. $$
By Proposition \ref{model},
$$ 
|\Delta_1(x)|^*\preceq4\frac{x^3}{2^3}+4\frac{x^4}{2^4}+9\sum\limits_{n\ge5}\frac{x^n}{2^n}.
$$
So for any $n\ge3$,
$$|\Delta_{1,n}|\le \min\{1/2, 9/2^{n}\}.$$

For any $n\ge 3$, $\Delta_{0,n}\le 1$, $\Delta_{1,n}\le 1/2<1$,
then by a same argument as above,
$$|\Delta_{2,n}|\le \min\{1/2, 9/2^{n}\}.$$

Continue this process, by Proposition \ref{model} and induction, for any $k\ge0$ and $n\ge3$,
$$\begin{array}{l}
|\Delta_{2k+1,n}|\le \min\{2^{-k-1}, (9\times 2^{-k})\times 2^{-n}\}\\
|\Delta_{2k+2,n}|\le \min\{2^{-k-1}, (9\times 2^{-k})\times 2^{-n}\}.
\end{array}$$
Recall that $\alpha=\sqrt{2}/2$. Thus  for any $n\ge3$,
$$|\Delta_{k,n}|\le 9\alpha^{k-2}\times 2^{-n},$$
which implies \eqref{pk}.
\end{proof}

Now we are ready for the proof of Theorem \ref{exp-cvge}. At first we define the absolute constants mentioned in that theorem.
Let  $\tilde C_0=9/(4-\pi)$ and for any $m\ge 0$ define
$$
\begin{cases}
C_m&=\tilde C_m(2^{m-1}\pi)^3 \\
\tilde C_{m+1}&=\tilde C_m\left(\frac{1}{2\alpha}+\frac{(4+C_m)^2}{64\alpha^2}\right).
\end{cases}
$$

\noindent {\bf Proof of Theorem \ref{exp-cvge}}.\  We prove it by induction on $m$.

At first consider $m=0.$  Fix $x\in[-\pi/2,\pi/2]$.
By Proposition \ref{control-coefficient}, for $k\ge 1$,
$$
|\Delta_k(x)|\le 9\alpha^{k-2}\sum_{n=3}^\infty\frac{|x|^n}{2^n}=
\frac{9\alpha^{k-2}|x|^3}{8-4|x|}\le \frac{9\alpha^k|x|^3}{4-\pi}=
\tilde C_0\alpha^k|x|^3\le   C_0\alpha^k.
$$

 Next we assume the result holds for $m<n$. Take  $x\in[-2^{n-1}\pi,2^{n-1}\pi]$,
 then $x/2,x/4\in[-2^{n-2}\pi,2^{n-2}\pi]$.
 For any $k\ge2n+1$, we have $k-1,k-2\ge 2n-1.$ By \eqref{delta-k-x} and induction,
\begin{eqnarray*}
|\Delta_k(x)|&=&|4\cos^2 \frac{x}{4}\cdot \Delta_{k-1}(\frac{x}{2})+\\
&&\Delta_{k-2}(\frac{x}{4})\Big(4\cos \frac{x}{4}+\Delta_{k-2}(\frac{x}{4})\Big)
\Big(2(\cos \frac{x}{2}-1)+\Delta_{k-1}(\frac{x}{2})\Big)|\\
&\le&4|\Delta_{k-1}(\frac{x}{2})|+|\Delta_{k-2}(\frac{x}{4})|(4 +|\Delta_{k-2}(\frac{x}{4})|)
(4+|\Delta_{k-1}(\frac{x}{2})|)\\
&\le&\frac{\tilde C_{n-1} \alpha^{k-1}}{2}|x|^3+\frac{\tilde C_{n-1}
\alpha^{k-2}}{64}|x|^3(4+C_{n-1})^2\\
&\le&\tilde C_{n-1} \left(\frac{1 }{2\alpha}+
\frac{ (4+C_{n-1})^2 }{64\alpha^2}\right)\alpha^k|x|^3\\
&=&\tilde C_n\alpha^k|x|^3\le C_n\alpha^k.
\end{eqnarray*}
 By induction the proof is finished.
 \hfill $\Box$


\section{ Generate new germs and control the distance of base points }\label{sec-3}

In this section we prepare the proof of Theorem \ref{main-bd}. Especially we will show that under some condition, new germs will appear and we can control the distance of the base points of the germs.

\subsection{Birth of new germs}\

In this subsection,
we exhibit  that how a regular germ of one pair can give birth to 
some new regular germs for the iterated pairs.
At first we prove several preliminary results.  Define   $\delta_0=0.01$, $\delta_1=0.0005$ and $\delta_2=10^{-10}$.

\begin{lem}\label{zero}
Fix $\delta\in(0,\delta_0)$ and  $k\in \mathbb{Z}.$ Assume $\varphi$ is a polynomial satisfying
$$
|\varphi(x)-2\cos x|\le \delta,\quad \forall x\in2k\pi+[0,\pi]\ (\mbox{resp. } 2k\pi+[-\pi,0]).
$$
We further assume $x_+$ (resp. $x_-$) is the minimal $x\in2k\pi+[0,\pi]$
(rsp. maximal $x\in2k\pi+ [-\pi,0]$) such that
$\varphi(x)=0$. Then
$$
|x_+-(2k\pi+\pi/2)|\le \delta\quad (\mbox{resp. }|x_--(2k\pi-\pi/2)|\le \delta ).
$$

\end{lem}

\begin{proof}
We only prove the result for  $x_+$ when  $k=0$ since the other cases can be proven similarly.

By the assumption we have
$$
|2\cos x_+|=|\varphi(x_+)-2\cos x_+| \le \delta.
$$
On the other hand since $|\sin x|\ge 2|x|/\pi$ for $|x|\le \pi/2$ and $|2k\pi+\frac{\pi}{2}-x_+|\le \frac{\pi}{2}$
$$
|2\cos x_+|=2|\sin(2k\pi+\frac{\pi}{2}-x_+)|\ge \frac{4}{\pi}|x_+-\frac{\pi}{2}-2k\pi|.
$$
This prove the lemma.
\end{proof}

\begin{cor}\label{zero-1}
Fix $\delta\in(0,\delta_0)$, $k\in \mathbb{Z}$ and $m\in\mathbb{N}.$ Assume $\varphi$ is a polynomial such that for any $x\in2^{1-m}k\pi+[0,2^{-m}\pi]$ (resp. $ 2^{1-m}k\pi+[-2^{-m}\pi,0])$
$$
|\varphi(x)-2\cos2^m x|\le \delta.
$$
We further assume $x_+$ (resp. $x_-$) is the minimal $x\in2^{1-m}k\pi+[0,2^{-m}\pi]$
(rsp. maximal $x\in2^{1-m}k\pi+ [-2^{-m}\pi,0]$) such that
$\varphi(x)=0$. Then
$$
|x_+-2^{-m}(2k\pi+\pi/2)|\le 2^{-m}\delta\quad (\mbox{resp. }|x_--2^{-m}(2k\pi-\pi/2)|\le 2^{-m}\delta ).
$$
\end{cor}

\proof\ Define $\tilde \varphi(x)=\varphi(2^mx)$, then apply Lemma \ref{zero}, the result follows.
\hfill $\Box$

\begin{cor}\label{ratio-change}
Let $\delta\in(0,\delta_0)$. Assume $\varphi,\psi$ are two polynomials satisfies
$$
|\varphi(x)-2\cos x|,\ \ \  |\psi(x)-2\cos 8x|\le \delta \ \ \ \forall x\in[0,\pi].
$$
Let $x^\ast$ be the minimal $x\in[0,\pi]$ such that
$\varphi(x)=0$ and $x_\ast$ be the maximal $x\in(0,x^\ast)$
such that $\psi(x)=0$. Then
$$
  |(x^\ast-x_\ast)-\pi/16|\le 2\delta.
$$
\end{cor}

\proof\
By Corollary \ref{zero-1} we have
$$
|x^\ast-\pi/2|\le \delta\ \ \ \ \text{ and }\ \ \ \ |x_\ast-7\pi/16|\le 2^{-3}\delta.
$$
Thus the result follows.
\hfill $\Box$

\begin{prop}\label{ratio}
Take any $\delta\le\delta_1$.
Assume  polynomial pair $(P_{-1},P_0)$  has a
$(\delta,2)$-regular $\rho$-germ at $x_0$. Then
there exist $y^-_{-1}<y^-_0<x_0<y^+_0<y^+_{-1}$  such that
$$
P_k(y^-_k)=P_k(y^+_k)=0; \ \  P_k(x)>0 \ \ (x\in I_k^-\cup I_k^+), \ \ \  (k=-1,0).
$$
where $I^-_k=(y^-_k,x_0]$, $I^+_k=[x_0,y^+_k)$.  Moreover
$$
\frac{|I^+_{-1}|}{|I^+_0|}, \ \ \ \frac{|I^-_{-1}|}{|I^-_0|}\le 2.1.
$$
\end{prop}

 \proof\
For $k=-1,0$, define $Q_k(x)$ as in \eqref{iterscale}.
By the assumption $(P_{-1},P_0)$ is $(\delta,2)$-regular at $x_0$, for $|x|<2$,
 $$
 |Q_{-1}(x)-2\cos x|, \ \ |Q_{0}(x)-2\cos x|\le \frac{|x|^3}{8-4|x|}\delta.
 $$
 Especially for $|x|\le 1.9$, we have
 $$
 |Q_{-1}(x)-2\cos x|, \ \ |Q_{0}(x)-2\cos x|\le \frac{1.9^3\delta}{8-4\cdot 1.9}<
 20\delta\le 0.01=\delta_0.
 $$
Let $t_0$ be the minimal $t\in (0,2)$ such that $Q_0(t)=0$,
$t_{-1}$ be the minimal $t\in(0,2)$  such that $Q_{-1}(t)=0$.
Then by Lemma \ref{zero},
$$
|t_0-\frac{\pi}{2}|\le \delta_0, \quad |t_{-1}-\frac{\pi}{2}|\le \delta_0,
$$
Define
$$
y^+_0=\frac{t_0}{a}+x_0,\quad  y^+_{-1}=\frac{2t_{-1}}{a}+x_0,
$$
we see $P_0(y^+_0)=P_{-1}(y^+_{-1})=0$ and
 $$
 \frac{|I^+_{-1}|}{|I^+_0|}=\frac{|y_{-1}^+-x_0|}{|y_0^+-x_0|}=\frac{2t_{-1}}{t_0}\le
 \frac{2(\frac{\pi}{2}+0.01)}{\frac{\pi}{2}-0.01}<2.1.
 $$
The proof of the  other part of the proposition is analogous.
 \hfill $\Box$

 Now we can state the main result in this subsection.

 \begin{prop}\label{1-1-regular}
Take any $\delta\le \delta_2$.
Assume  $(P_{-1},P_0)$ is
$(\delta,2)$-regular at $x_0$ with renormalization factor $a$, define $P_n=P_{n-2}^2(P_{n-1}-2)+2$ for $n\ge 1$.
Let  $y^+_0,y^-_0$ be defined as in Proposition \ref{ratio}, then
$(P_3,P_4)$ is $(1,1)$-regular at both $y^{+}_0$ and $y^-_0$.
\end{prop}

\begin{proof}
  Recall that we have defined
  \begin{equation}\label{q-k}
Q_k(x)=P_k(\frac{x}{2^ka}+x_0)=:2\cos x+\Delta_k(x),\quad k=-1,0.
\end{equation}
Then
$$
|\Delta_{-1}(x)|^\ast \preceq \delta\sum_{n=3}^\infty \frac{x^n}{2^n},\quad
|\Delta_0(x)|^\ast \preceq \delta\sum_{n=3}^\infty \frac{x^n}{2^n}.
$$
Consequently for $x\in(0,2) $ it is ready to show that
 \begin{equation}\label{derivatives}
 \begin{cases}
 |\Delta_0^{(n)}(x)|\\
  |\Delta_{-1}^{(n)}(x)|
 \end{cases}
 \le \delta \left(\sum_{n=3}^\infty\frac{x^n}{2^n}\right)^{(n)}=
 \frac{\delta}{4}\left(\frac{x^3}{2-x}\right)^{(n)}=
 \begin{cases}
 \frac{\delta}{4}\left(\frac{x^3}{2-x}\right)& n=0\\
 -\frac{\delta(x+1)}{2}+\frac{2\delta}{(2-x)^2}& n=1\\
-\frac{\delta}{2}+\frac{4\delta}{(2-x)^3} & n=2\\
 \frac{2\delta n!}{(2-x)^{n+1}}& n\ge 3.
 \end{cases}
 \end{equation}

Let $t_0$ be the minimal $t\in (0,2)$ such that $Q_0(t)=0$.
Then
\begin{equation}\label{est-t0}
P_0(t_0/a+x_0)=0,\quad |t_0-\frac{\pi}{2}|\le 0.01.
\end{equation}
Consequently  $y^+_0=t_0/a+x_0$.  We have
$$
P_{-1}(y_0^+)=Q_{-1}(\frac{t_0}{2})=2\cos \frac{t_0}{2}+\Delta_{-1}(\frac{t_0}{2}).
$$
By \eqref{derivatives} and \eqref{est-t0} we get $|P_{-1}(y_0^+)-\sqrt{2}|\le 0.03.$ Since 
$$
P_{1}(y_0^+)=P_{-1}^2(y_0^+)(P_0(y_0^+)-2)+2,
$$ 
we conclude that
\begin{equation}\label{P1}
|P_1(y_0^+)+2|\le 0.2.
\end{equation}
By \eqref{q-k} we have $P_0^\prime(y_0^+)=aQ_0^\prime(t_0)=a(-2\sin t_0+\Delta_0^\prime(t_0))$. By \eqref{derivatives} and \eqref{est-t0} we get
\begin{equation}\label{P0}
\left|\frac{P_0^\prime(y_0^+)}{a}+2\right|\le 0.2.
\end{equation}
By \eqref{fkx} and \eqref{P1}, \eqref{P0}, for $k=3,4$ we have
$$
P_k(x)=2-(2^{k-3}\rho)^2(x-y_0^+)^2 +O((x-y_0^+)^3)
$$
with
$$
\rho=\sqrt{2-P_1(y_0^+)}|P_0^\prime (y_0^+)P_1(y_0^+)|\ge 6a.
$$
For $k\ge 3$, if we define $\tilde Q_k(x):= P_k(\frac{x}{2^{k-3}\rho}+y_0^+),$ then
\begin{equation}\label{tildeqk}
\tilde Q_k(x)=2-x^2+O(x^3)=2\cos x+O(x^3)=: 2\cos x +\bar \Delta_k(x).
\end{equation}
We need to show that
$$
|\bar\Delta_3(x)|^\ast, |\bar\Delta_4(x)|^\ast \preceq \sum_{n\ge 3} x^n.
$$

For $k\ge -1$ define $\tilde \Delta_k(2^kx):= \Delta_k(2^k(x+t_0))$. Then
\begin{equation}\label{pk-1}
 P_k(\frac{x}{a}+y_0^+)=P_k(\frac{x+t_0}{a}+x_0)=:2\cos(2^k(x+t_0))+\tilde\Delta_k(2^k x).
\end{equation}

By the recurrence relation of $P_k$, it is ready to show that $(\tilde \Delta_k(x))_{k\ge -1}$ satisfies \eqref{delta-k-x-1}. We have
 $$
 \begin{cases}
 \tilde \Delta_0(x)&=\Delta_0(x+t_0)=\sum_{n=0}^\infty \frac{\Delta_0^{(n)}(t_0)}{n!}x^n\\
 \tilde \Delta_{-1}(x)&=\Delta_{-1}(x+t_0/2)=\sum_{n=0}^\infty \frac{\Delta_{-1}^{(n)}(t_0/2)}{n!}x^n.
 \end{cases}
$$
Write $\beta=2-\pi/2-0.01=0.419\cdots$ and $M_0=10.$ By \eqref{derivatives} and \eqref{est-t0} we get
$$
 |\tilde \Delta_0(x)|^\ast,  |\tilde \Delta_{-1}(x)|^\ast\preceq M_0\delta\sum_{n=0}^\infty \frac{x^n}{\beta^n}.
$$
Since $\delta\le \delta_2=10^{-10}$, we have $152^3M_0\delta<1.$ By Proposition \ref{model2}
$$
|\tilde \Delta_1(x)|^\ast\preceq 152 M_0\delta\sum_{n=0}^\infty \frac{x^n}{(2\beta)^n}.
$$
Consequently we have
$$|\tilde\Delta_0(x)|^\ast\preceq 152M_0\delta\sum_{n=0}^\infty \frac{x^n}{\beta^n},\quad
|\tilde \Delta_1(x)|^\ast\preceq 152 M_0\delta\sum_{n=0}^\infty \frac{x^n}{\beta^n}.
$$
Again by Proposition \ref{model2} we get
$$|\tilde \Delta_2(x)|^\ast\preceq 152^2 M_0\delta\sum_{n=0}^\infty \frac{x^n}{(2\beta)^n}.
$$
Continue this process, since $152^3 M_0\delta<1$ and $2\beta<1$,  for $k=3,4$ we get
$$|\tilde \Delta_k(x)|^\ast\preceq 152^k M_0\delta\sum_{n=0}^\infty \frac{x^n}{(4\beta)^n}.$$
By these inequalities and \eqref{pk-1}, \eqref{cosine}, we have, for $k=3,4$,
$$
|P_k(\frac{x}{a}+y_0^+)|^\ast\preceq \sum_{n\ge0}\frac{(2^kx)^n}{(4\beta)^n}
\left(\frac{2(4\beta)^n}{n!}+152^kM_0\delta\right).
$$
Notice that $\tilde Q_k(x)=P_k(\frac{ax/2^{k-3}\rho}{a}+y_0^+)$ and $\rho\ge 6a$, we get
$$
|\tilde Q_k(x)|^\ast\preceq \sum_{n\ge0}\frac{(4x/3)^n}{(4\beta)^n}
\left(\frac{2(4\beta)^n}{n!}+152^kM_0\delta\right).
$$
Combine with \eqref{tildeqk} we have
$$\begin{array}{rcl}
|\tilde Q_k(x)-2+x^2|^\ast&\preceq&
\sum_{n\ge3}\frac{(4x/3)^n}{(4\beta)^n}
\left(\frac{2(4\beta)^n}{n!}+152^kM_0\delta\right)\\
|\tilde Q_k(x)-2\cos x|^\ast&\preceq&
\sum_{n\ge3}\frac{(4x/3)^n}{(4\beta)^n}
\left(\frac{2(4\beta)^n}{n!}+152^kM_0\delta\right)+
\sum_{n\ge4}\frac{2x^n}{n!}.
\end{array}$$
Now by a direct computation, for $k=3,4$,
$$
|\bar\Delta_k(x)|^\ast=|\tilde Q_k(x)-2\cos x|^* \preceq
\sum_{n\ge3}x^n.$$
This prove   the result for $y_0^+$. The proof for $y_0^-$ is the same.
\end{proof}


\subsection{ Control the distance between the base points of  germs }\

In this subsection we will show that besides the birth of new germs, we can even control the distance between the base points of the old and new germs once we iterate sufficient long time.

At first we define a big absolute integer. Let  $n_\alpha=40,$ then $9\alpha^{n_\alpha-3}\le \delta_1$. Define $\delta_3:=30^{-n_\alpha}/4000.$
Choose an absolute constant $K\in\mathbb{N}$ such that
\begin{equation}\label{def-K}
K\ge n_\alpha+4,\ \ \
9\alpha^{K-7}<\delta_2,\ \ \
C_6\alpha^{K-4}\le \delta_3,
\end{equation}
where $C_6$ is an absolute constant defined in Theorem  \ref{exp-cvge}.

We need one more lemma.

\begin{lem}\label{n-alpha}
Fix $N\ge 2$ and  $\delta>0$. Assume
$\varphi_0(x),\varphi_1(x)$ are two polynomials satisfying
$$
\begin{cases}
|\varphi_0(x)-2\cos 8x|\le \delta & x\in[-\pi,\pi]\\
|\varphi_1(x)-2\cos 16x|\le \delta & x\in[-\pi/2,\pi/2].
\end{cases}
$$
Define $\varphi_{n}=\varphi_{n-2}^2(\varphi_{n-1}-2)+2$ for $n\ge 2.$
Then for any $2\le n\le N$ and $x\in[-\pi/2^n,\pi/2^n]$
\begin{equation}\label{err}
|\varphi_n(x)-2\cos 2^{n+3}x|\le 30^{n-1}\delta.
\end{equation}
\end{lem}

\proof\
Write $\varphi_n(x)=2\cos 2^{n+3}x+\Delta_n(2^{n+3}x)$ for $n\ge 0.$ It is equivalent to prove that for any $n\ge 2$
\begin{equation}\label{pro-equi}
|\Delta_n(x)|\le 30^{n-1}\delta\ \ \ (\forall x\in [-8\pi,8\pi]).
\end{equation}

We prove it by induction.
The condition implies that
\begin{equation}\label{delta01}
|\Delta_0(x)|, |\Delta_1(x)|\le \delta\ \  (x\in[-8\pi,8\pi]).
\end{equation}
By the recurrence relation, similar with  \eqref{delta-k-x}, for $n\ge 2$ we have
\begin{eqnarray*}
\Delta_n(2^{n+3}x)&=&(2+2\cos 2^{n+2}x)\cdot\Delta_{n-1}(2^{n+2}x)+\Delta_{n-2}(2^{n+1}x)\cdot\\
\nonumber&&\Big(4\cos 2^{n+1}x+\Delta_{n-2}(2^{n+1}x)\Big)
\Big(2\cos 2^{n+2}x-2+\Delta_{n-1}(2^{n+2}x)\Big).
\end{eqnarray*}
For any $x\in[-\pi/2^2,\pi/2^2]$, we have $2^3x,2^4x\in [-4\pi,4\pi]$. Thus
by \eqref{delta01}
$$
|\Delta_2(2^5x)|\le 4\delta+\delta(4+\delta)^2\le 4\delta+25\delta<30\delta,
$$
i.e., \eqref{pro-equi} holds for $n=2$.

Now assume  \eqref{err} holds for $n<m\le N$.  For any $x\in [-\pi/2^m,\pi/2^m]$, we have $2^{m+1}x, 2^{m+2}x\in [-4\pi,4\pi]$.
Then
\begin{eqnarray*}
|\Delta_m(2^{m+3}x)|&\le& 4|\Delta_{m-1}
(2^{m+2}x)|+|\Delta_{m-2}(2^{m+1}x)|\cdot\\
&&(4+|\Delta_{m-2}(2^{m+1}x)|)(4+|\Delta_{m-1}(2^{m+2}x)|)\\
&\le&4\cdot 30^{m-2}\delta+30^{m-3}\delta(4+30^{m-3}\delta)(4+30^{m-2}\delta)\\
&\le&4\cdot 30^{m-2}\delta+25\cdot 30^{m-2}\delta\le 30^{m-1}\delta.
\end{eqnarray*}
Thus \eqref{pro-equi} holds for $n=m.$
This proves the lemma.
\hfill $\Box$

\begin{prop}\label{key-lem}
Assume $(P_{-1},P_0)$ is $(1,1)$-regular at $\theta.$ Let $\theta^+>\theta $ be the minimal  zero of $P_{K-4}$ in $[\theta,\infty)$. Then $(P_{K-1},P_{K})$ is $(1,1)$-regular at $\theta^+.$ Moreover assume $\theta_+\in (\theta,\theta^+)$ is the maximal zero of $P_{2K-4},$ then
\begin{equation}\label{eq-0}
\frac{\theta^+-\theta_+}{\theta^+-\theta}\ge 2.1^{-K}.
\end{equation}
Similarly let $\theta_-<\theta $ be the maximal  zero of $P_{K-4}$ in $[-\infty,\theta)$. Then $(P_{K-1},P_{K})$ is $(1,1)$-regular at $\theta_-.$ Moreover assume $\theta^-\in (\theta_-,\theta)$ is the minimal zero of $P_{2K-4},$ then
$$
\frac{\theta^--\theta_-}{\theta-\theta_-}\ge 2.1^{-K}.
$$
\end{prop}

\begin{proof} We only prove the first result, since the second one is the same.

Since $(P_{-1},P_0)$ is (1,1)-regular at $\theta$, by Proposition \ref{control-coefficient}, $(P_{K-5,K-4})$ is $(9\alpha^{K-7},2)$-regular at $\theta.$ By \eqref{def-K}, $(P_{K-5,K-4})$ is $(\delta_2,2)$-regular at $\theta.$ Then by Proposition \ref{1-1-regular}, $(P_{K-1},P_K)$ is $(1,1)$-regular at $\theta^+.$

For any $-1\le k\le K-4$ define $\theta_k$ to be the largest $y$ in
$[\theta,\theta^+]$ such that $P_{K+k}(y)=0$, then $\theta_{K-4}=\theta_+$. We have
$$
\frac{\theta^+-\theta_+}{\theta^+-\theta}=\frac{\theta^+-\theta_{-1}}{\theta^+-\theta}\cdot \prod_{k=0}^{K-4} \frac{\theta^+-\theta_{k}}{\theta^+-\theta_{k-1}}
$$

At first we estimate $(\theta^+-\theta_{-1})/(\theta^+-\theta).$

Assume $(P_{-1},P_0)$ has renormalization factor $a$ at $\theta.$
let $L_k(x)=\frac{x}{2^{k}a}+\theta$. By Theorem  \ref{exp-cvge}, for any $k\ge 13$ and any $x\in[-32\pi,32\pi]$ we have
\begin{equation}\label{expansion_a_empty}
|P_k\circ L_k(x)-2\cos x|\le C_6\alpha^k.
\end{equation}

By \eqref{expansion_a_empty} and \eqref{def-K}, for  $x\in[-4\pi,4\pi]$,
$$ 
|P_{K-4}\circ L_{K-4}(x)-2\cos x| \le \delta_3,\quad
|P_{K-1}\circ L_{K-4}(x)-2\cos 8x| \le  \delta_3.
$$ 
By Lemma \ref{zero} and Corollary \ref{ratio-change},
\begin{equation}\label{p34}
|L_{K-4}^{-1}(\theta^+)-\pi/2|\le\delta_3\le \delta_0,\quad
|(L_{K-4}^{-1}(\theta^+)-L_{K-4}^{-1}(\theta_{-1}))-\pi/16|\le  2\delta_3\le \delta_0.
\end{equation}
Since $L_{k}^{-1}x=2^ka(x-\theta)$ we conclude that
\begin{equation}\label{eq1}
\frac{\theta^+-\theta_{-1}}{\theta^+-\theta}\ge 2.1^{-3}.
\end{equation}

Next we  we estimate $(\theta^+-\theta_{k})/(\theta^+-\theta_{k-1})$ for $k=0,\cdots, K-4.$ We will discuss two cases. We have shown that $(P_{K-1},P_K)$ is $(1,1)$-regular at $\theta^+.$ By Proposition \ref{control-coefficient}, $(P_{K+k-1},P_{K+k})$ is $(9\alpha^{k-3},2)$-regular.

In the following, we prove 
$\frac{\theta^+-\theta_{k}}{\theta^+-\theta_{k-1}}\ge 2.1^{-1}$ in two cases:
$$n_\alpha\le k\le K-4 \quad\mbox{and}\quad 0\le k<n_\alpha.$$

Case i: $n_\alpha\le k\le K-4.$

Recall that  $n_\alpha$ is such that $9\alpha^{n_\alpha-3}\le \delta_1$. In this case $(P_{K+k-1},P_{K+k})$ is $(\delta_1,2)$-regular, then by
Proposition \ref{ratio},
\begin{equation}\label{eq2}
\frac{\theta^+-\theta_{k}}{\theta^+-\theta_{k-1}}\ge 2.1^{-1}.
\end{equation}

Case ii: $0\le k<n_\alpha.$
By \eqref{expansion_a_empty} and \eqref{def-K},
$$
\begin{cases}
|P_{K-1}\circ L_{K-4}(x)-2\cos 8x| \le  \delta_3 & x\in[-4\pi,4\pi]\\
|P_{K}\circ L_{K-4}(x)-2\cos 16x| \le  \delta_3 & x\in[-2\pi,2\pi].
\end{cases}$$

Let $x^*=L_{K-4}^{-1}(\theta^+)$,
define $\xi_k(x):=P_{K+k-1}\circ L_{K-4}(x+x^\ast)$ for $k\ge 0$.

For $x\in[-\pi,\pi]$,
\begin{eqnarray*}
&&|\xi_{0}(x)-2\cos 8x|\\&=&|P_{K-1}\circ L_{K-4}(x+x^\ast)-2\cos 8x|\\
&\le&|P_{K-1}\circ L_{K-4}(x+x^\ast)-2\cos 8(x+x^\ast)|+|2\cos 8(x+x^\ast)-2\cos 8x|\\
&\le& \delta_3+4|\sin 4x^\ast|= \delta_3+4|\sin 4(x^\ast-\frac{\pi}{2})|\le 17\delta_3\le 30^{-n_\alpha}/100,
\end{eqnarray*}
where the third inequality is due to \eqref{p34}, and the last inequality is by definition
of $\delta_3$.

Similarly we have  for $x\in[-\pi/2,\pi/2]$
\begin{eqnarray*}
&&|\xi_1(x)-2\cos 16x|\\
&=&|P_{K}\circ L_{K-4}(x+x^\ast)-2\cos 16x|\\
&\le&|P_{K}\circ L_{K-4}(x+x^\ast)-2\cos 16(x+x^\ast)|+|2\cos 16(x+x^\ast)-2\cos 16x|\\
&\le& \delta_3+4|\sin 8x^\ast|\le \delta_3+4|\sin 8(x^\ast-\frac{\pi}{2})|\le 33\delta_3\le 30^{-n_\alpha}/100.
\end{eqnarray*}
By Lemma \ref{n-alpha}, For any $2\le k\le n_\alpha$ and  $x\in[-\pi/2^{k},\pi/2^{k}]$,
$$
|\xi_k(x)-2\cos 2^{k+3}x|\le 30^{k-1}\cdot 30^{-n_\alpha}/100\le 1/100=\delta_0.
$$
Notice that for all $-1\le k<n_\alpha,$
$L_{K-4}^{-1}(\theta_{k})-x^*=2^{K-4}a(\theta_{k}-\theta^+)$ is the largest zero point of $\xi_{k+1}$ in
$[-\frac{\pi}{2^{k+4}},0]$.
Then by Corollary \ref{zero-1},
$$
|2^{K-4}a(\theta_{k}-\theta^+)+\frac{\pi}{2^{k+5}}|\le \frac{1}{100\cdot 2^{k+4}}\le
\frac{1}{100}\frac{\pi}{2^{k+5}}.
$$
Consequently, for any integer $0\le k<n_\alpha$,
\begin{equation}\label{eq3}
\frac{\theta^+-\theta_k}{\theta^+-\theta_{k-1}}\ge 2.1^{-1}.
\end{equation}
Combining \eqref{eq1},\eqref{eq2} and \eqref{eq3} we get \eqref{eq-0}.
\end{proof}


\section{Lower bound for the Hausdorff dimension of the spectrum}\label{sec-4}

Now we go back to the spectrum of the Thue-Morse Hamiltonian. Recall that  $\{h_n(x):n\ge 1\}$ is the related trace polynomials and  $h_1(x)=x^2-\lambda^2-2.$ Let $a_\emptyset=\sqrt{2+\lambda^2},$ then  $h_1(a_\emptyset)=0.$

At first we will show that $(h_4,h_5)$ is $(1,1)$-regular at $a_\emptyset,$ then by Proposition \ref{key-lem} we will construct a Cantor subset of the spectrum around $a_\emptyset$ and estimate the dimension of the Cantor set. Consequently we get a lower bound for the Hausdorff dimension of the spectrum.

\subsection{$(h_4,h_5)$ is $(1,1)$-regular at $a_\emptyset$}\

Recall that $h_2(x)=(x^2-\lambda^2)^2-4x^2+2,$
 then
$h_2(a_\emptyset)=-2-4\lambda^2.$  Thus
$$
\begin{cases}
h_1(x)=h_1^\prime(a_\emptyset)(x-a_\emptyset)+O((x-a_\emptyset)^2)=
2a_\emptyset(x-a_\emptyset)+O((x-a_\emptyset)^2)\\
h_2(x)=h_2(a_\emptyset)+O((x-a_\emptyset))=-2(1+2\lambda^2)+O((x-a_\emptyset)).
\end{cases}
$$
Write $\rho:=(1+2\lambda^2)\sqrt{(1+\lambda^2)(2+\lambda^2)}$. By using the recurrent relation
we get
$$
h_k(x)=2-4^{k-1}\rho^2(x-a_\emptyset)^2+O((x-a_\emptyset)^3)\ \ (k\ge 4).
$$

\begin{lem}\label{initial-condition}
$(h_4,h_5)$ is $(1,1)$-regular at $a_\emptyset$ with renormalization factor $2^4\rho$.
\end{lem}

\begin{proof}
  Write $t:=2a_\emptyset$. Then $t\ge 2\sqrt{2}.$ Define
$$
g_n(x)=h_n(x+a_\emptyset),
$$
we have
$$
g_1(x)=x^2+t x\ \ \ \ \text{ and }\ \ \ \  g_2(x)=x^4+2t x^3+t^2 x^2+6-t^2.
$$
Then we can compute that for $n\ge 4$,
$$
g_n(x)=2-4^{n-4}t^2(t^2-6)^2(t^2-4) x^2+O(x^3).
$$
Write $\tau:= t(t^2-6)\sqrt{t^2-4}$. Define $f_n(x)=g_n(x/\tau)$. Then for $n\ge 4$ we have
$$
f_n(x)=2-4^{n-4}x^2+O(x^3).
$$
We also have
\begin{eqnarray*}
f_1(x)&=&{t}{\tau}^{-1}x+{\tau^{-2}}x^2\\
f_2(x)&=&(6-t^2)+(t/\tau)^2x^2+(2t/\tau^3)x^3+x^4/\tau^4\\
f_3(x)&=&2-x^2/(t^2-4)-O(x^3).
\end{eqnarray*}
By the fact that $t\ge 2\sqrt{2}$ and $\tau=t(t^2-6)\sqrt{t^2-4}$, it is direct to verify that
$$|f_1(x)|^*\preceq\frac{t}{\tau}xe^{x/32},\quad
|f_2(x)|^*\preceq(t^2-6)e^{x/4},\quad |f_2(x)-2|^*\preceq (t^2-4)e^{x/4}.
$$
Then, by $f_n=2+f_{n-2}^2(f_{n-1}-2)$, we have,
$$\begin{array}{l}
|f_3(x)|^*\preceq 2+\frac{1}{(t^2-6)^2}x^2e^{5x/16}\\
|f_4(x)|^*\preceq 2+x^2e^{13x/16}\\
|f_5(x)|^*\preceq 2+4x^2e^{13x/16}+4\frac{x^4}{(t^2-6)^2}e^{18x/16}+\frac{x^6}{(t^2-6)^4}e^{23x/16}.
\end{array}
$$
Now, it is ready to verify that
$$|f_4(x)-2\cos x|^*\preceq \sum_{n\ge3}x^n,\quad
|f_5(x/2)-2\cos x|^*\preceq \sum_{n\ge3}x^n.$$
This implies that $(h_4,h_5)$ is $(1,1)$-regular at $a_\emptyset$
with renormalization factor $2\tau=2^4\rho$.
\end{proof}

\subsection{Lower bound of the spectrum}\

Now we will construct the desired Cantor set. To simplify the notation we write $P_k(x):=h_{k+5}(x)$ for $k\ge -1.$ Then we have shown that $(P_{-1},P_0)$ is $(1,1)$-regular at $a_\emptyset.$

Fix an absolute constant $K\in\mathbb{N}$ according to \eqref{def-K}.
Assume $b_\emptyset>a_\emptyset$ is the first zero of $P_{K-4}$ to
the right of $a_\emptyset.$ Define $I_\emptyset:=[a_\emptyset,b_\emptyset]$.
Assume $b_0$ is the smallest zero of $P_{2K-4}$ in $I_\emptyset$ and $a_1$
is the biggest zero of $P_{2K-4}$ in $I_\emptyset.$ Define
$$
I_0:=[a_\emptyset,b_0]=[a_0,b_0]\ \ \ \  \text{ and } \ \ \ I_1:=[a_1,b_\emptyset]=[a_1,b_1].
$$
Take any $w\in\{0,1\}^k$, suppose $I_w=[a_w,b_w]$ is defined.
 Assume $b_{w0}$ is the smallest zero of $h_{(k+2)K-4}$ in $I_w$ and
 assume $a_{w1}$ is the biggest zero of $h_{(k+2)K-4}$ in $I_w$.
 Write $a_{w0}=a_w$, $b_{w1}=b_w$ and define
 $$
 I_{w0}=[a_w,b_{w0}]=[a_{w0}, b_{w0}]\ \ \ \text{ and }\ \ \  I_{w1}
 =[a_{w1},b_{w}]=[a_{w1}, b_{w1}].
 $$
 Define  a Cantor set as
 $$
 \mathcal C:=\bigcap_{n\ge 1}\bigcup_{|w|=n}I_w
 $$

 \begin{prop}
  $$
 \dim_H \mathcal C \ge \frac{\ln 2}{K\ln 2.1}.
 $$
 \end{prop}

\begin{proof} Given $w\in \{0,1\}^{k}$. By induction it is easy to show  that
$$
\{P_{(k+1)K-4}(a_w),P_{(k+1)K-4}(b_w)\}=\{0,2\}.
$$
For definiteness let   $\alpha_w$ be  the endpoint of $I_w$ such that $P_{(k+1)K-4}(\alpha_w)=2$ and $\beta_w$ be another endpoint of $I_w$, thus $P_{(k+1)K-4}(\beta_w)=0.$

\noindent {\bf Claim:}    $(P_{kK-1}, P_{kK})$ is $(1,1)$-regular at $\alpha_w$.

\noindent $\lhd$ We show it by induction on $k$. When $k=0$, $w=\emptyset$. By Lemma \ref{initial-condition}, $(P_{-1},P_0)$ is $(1,1)$ regular at  $\alpha_w=a_w=a_\emptyset.$

Assume the result holds for any $w\in \{0,1\}^l$ with $l<k$. Now fix $w\in \{0,1\}^k$, then $w=\bar{w}i$ with $\bar{w}\in \{0,1\}^{k-1}$ and $i\in \{0,1\}.$ Then $\alpha_w$ must be one of the endpoints of $I_{\bar{w}}$ since $P_{(k+1)K-4}(\alpha_w)=2$. By the induction assumption $(P_{(k-1)K-1},P_{(k-1)K})$ is $(1,1)$-regular at $\alpha_{\bar{w}}$. Consequently $(P_{kK-1},P_{kK})$ is $(9\alpha^{K-3},2)$-regular at $\alpha_{\bar{w}}$. By \eqref{def-K}, $(P_{kK-1},P_{kK})$ is $(1,1)$-regular at $\alpha_{\bar{w}}$. On the other hand by Proposition \ref{key-lem}, we have $(P_{kK-1},P_{kK})$ is $(1,1)$-regular at $\beta_{\bar{w}}$.
Since $\alpha_w=\alpha_{\bar{w}}$ or $\beta_{\bar{w}}$, the result follows.
\hfill $\rhd$

Now fix any $w\in \{0,1\}^{k}$, let us estimate $|I{wi}|/|I_w|$ for $i=0,1$.
Without loss of generality we assume $\alpha_w=a_w$ and $\beta_w=b_w.$
By the claim above $(P_{kK-1},P_{kK})$ is $(1,1)$-regular at $a_w$.
By Proposition \ref{control-coefficient}, $(P_{kK+l-1},P_{kK+l})$ is $(9\alpha^{l-3},2)$-regular for any $l\ge 2$. By \eqref{def-K}, we have $9\alpha^{l-3}<\delta_1$ for any $l\ge K-4$. Thus $(P_{kK+l-1},P_{kK+l})$ is  $(\delta_1,2)$-regular for any $l\ge K-4$. Let $y_l$ be the smallest zero of $P_{kK+l}$ in $[a_w,\infty)$, then $y_{K-4}=b_w$ and $y_{2K-4}=b_{w0}.$ by
Proposition \ref{ratio}, for any $l\ge K-4$
$$
\frac{y_{l+1}-a_w}{y_{l}-a_w}\ge 2.1^{-1}.
$$
Consequently we have
$$
\frac{|I_{w0}|}{|I_w|}=\frac{b_{w0}-a_w}{b_w-a_w}=\prod_{l=K-4}^{2K-5}\frac{y_{l+1}-a_w}{y_{l}-a_w}\ge 2.1^{-K}.
$$
On the other hand, since $(P_{kK-1},P_{kK})$ is $(1,1)$-regular at $a_w$,
 by Proposition \ref{key-lem}, $(P_{(k+1)K-1},P_{kK})$ is $(1,1)$-regular at $b_w$. Moreover since $a_{w1}$ is the  maximal zero of $P_{(k+2)K-4}$ in $I_w$,
 $$
\frac{|I_{w1}|}{|I_w|}=\frac{b_{w}-a_{w1}}{b_w-a_w} \ge 2.1^{-K}.
$$

 Now it is  well known  that (see for example \cite{F})
$$
\dim_H \mathcal C\ge \frac{\ln 2}{-\ln 2.1^{-K}}=\frac{\ln2}{K\ln 2.1}.
$$
\end{proof}

 \begin{proof}[Proof of Theorem \ref{main-bd}]
Recall that  $\Sigma$ is defined  in Remark \ref{rem-zero}, thus
for any $w\in\{0,1\}^*$, $a_w,b_w\in\Sigma\subset \sigma(H_{\lambda,v}).$
Since $\mathcal C=\overline{\{a_w,b_w: w\in \{0,1\}^* \}},$  we conclude that $\mathcal C \subset \sigma(H_{\lambda,v})$.
Consequently
 $$
 \dim_H\sigma(H_{\lambda,v}) \ge \dim_H \mathcal C \ge \frac{\ln 2}{K\ln 2.1}.
 $$
 Since $K$ is an absolute positive constant, the result follows.
 \end{proof}

\noindent
{\bf Acknowledgements}. Liu and Qu are supported by the National Natural Science Foundation of China, No. 11371055.  Qu is supported by the National Natural Science Foundation of China, No. 11201256.



\begin{thebibliography}{30}

 \bibitem{AP1} Axel, F.,  Peyri\`ere, J.:  Extended states in a chain with controlled disorder.  C. R. Acad. Sci. Paris SŽr. II MŽc. Phys. Chim. Sci. Univers Sci. Terre 306, 179-182 (1988)

\bibitem{AP} Axel, F.,   Peyri\`ere, J.:  Spectrum and extended states in a harmonic
chain with controlled disorder: Effects of the Thue-Morse symmetry. 
 J. Statist. Phys. 57, 1013-1047 (1989)

\bibitem{B} Bellissard, J.: Spectral properties of Schr\"oinger operator with a
Thue-Morse potential.  In: Number
Theory and Physics (Les Houches, 1989),   Springer Proc. Phys. 47, pp. 140-150. Springer, Berlin  (1990)


\bibitem{BG}  Bovier, A.,   Ghez, J. M.: Spectral properties of one-dimensional
Schr\"oinger operators with potentials generated by substitutions. 
Commun. Math. Phys. 158, 45-66 (1993)  

\bibitem{C}Cantat, S.: Bers and H\'enon, Painlev\'e and Schr\"odinger. Duke Math. J. 149, 411-460  (2009)



\bibitem{DEGT} Damanik, D.,  Embree,  M., Gorodetski, A., 
Tcheremchantsev, S.:  The fractal dimension of the spectrum of the
Fibonacci Hamiltonian.  Commun. Math. Phys. 280, 499-516 (2008)

\bibitem{DG}Damanik, D.,   Gorodetski, A.: Hyperbolicity of the trace map for the weakly coupled Fibonacci
Hamiltonian.  Nonlinearity 22, 123-143  (2009)

\bibitem{DG2} Damanik,  D.,  Gorodetski, A.: Spectral and quantum dynamical properties of the weakly coupled Fibonacci Hamiltonian. Commun. Math. Phys. 305, 221-277 (2011)

 
\bibitem{F} Falconer, K.:  Fractal geometry. Mathematical foundations and applications.  John Wiley \& Sons, Ltd., Chichester (1990)

\bibitem{JL} Jitomirskaya, S.,   Last, Y.: Power-law subordinacy and singular spectra. II. Line operators. 
Commun. Math. Phys. 211, 643-658 (2000)

\bibitem{Lenz} Lenz, D.: Uniform ergodic theorems on subshifts over a finite alphabet. 
Ergodic Theory Dynam. Systems 22, 245-255 (2002)

\bibitem{LTWW} Liu, Q. H.,   Tan, B.,  Wen, Z. X.,   Wu, J.: Measure zero spectrum of a class of
Schr\"odinger operators.  J. Statist. Phys. 106, 681-691 (2002)

\bibitem{LW} Liu, Q. H.,  Wen, Z. Y.: 
Hausdorff dimension of spectrum of one-dimensional
Schr\"odinger operator with Sturmian potentials. 
 Potential Analysis 20, 33-59 (2004)

\bibitem{R} Raymond, L.: 
A constructive gap labelling for the discrete schr\"odinger
operater on a quasiperiodic chain. (Preprint,1997)

 
 \end{thebibliography}
\end{document}